\documentclass[10pt]{amsart}
\usepackage[utf8]{inputenc}
\usepackage{amsmath,amssymb,amsthm}
\usepackage[colorlinks,citecolor=blue]{hyperref}
\usepackage{cleveref}
\usepackage{mathrsfs}
\usepackage{graphicx}

\numberwithin{equation}{section} \overfullrule 5pt
\newtheorem{Theorem}{Theorem}[section]

\newtheorem{Proposition}[Theorem]{Proposition}
\newtheorem{Conjecture}[Theorem]{Conjecture}
\newtheorem{Lemma}[Theorem]{Lemma}

\theoremstyle{definition}

\newcommand{\oeis}[1]{\href{https://oeis.org/#1}{{#1}}}
\newcommand{\next}[1]{\mathscr{N}_{#1}}
\newcommand{\vs}[1]{\mathsf{V}_{#1}}
\newcommand{\cs}[1]{\mathtt{#1}}
\newcommand{\epsi}{\varepsilon}

\title{On the roots of the Poupard\\ and Kreweras Polynomials}
\date{\today}

\author{Frédéric Chapoton and Guo-Niu Han}
\address{Institut de Recherche Mathématique Avancée, 
UMR 7501, Université de Strasbourg et CNRS, 
7 rue René Descartes, 
67000 Strasbourg, France}
\email{chapoton@unistra.fr, \quad guoniu.han@unistra.fr}

\begin{document}

\maketitle

\begin{abstract}
  The Poupard polynomials are polynomials in one variable with integer
  coefficients, with some close relationship to Bernoulli and tangent
  numbers. They also have a combinatorial interpretation. We prove
  that every Poupard polynomial has all its roots on the unit
  circle. We also obtain the same property for another sequence of
  polynomials introduced by Kreweras and related to Genocchi
  numbers. This is obtained through a general statement about some
  linear operators acting on palindromic polynomials.
\end{abstract}

\section{Introduction}

Let us consider the sequence of polynomials $(F_n)_{n \geq 1}$ in one
variable $x$ characterized by the equation
\begin{equation}
  \label{recu_F}
  (x - 1)^2 F_{n+1}(x) = (x^{2n+2} + 1) F_{n}(1) - 2 x^2 F_n(x) \quad \text{for} \quad  n \geq 1,
\end{equation}
with the initial condition $F_1 = 1$. When described in this way,
their existence is not completely obvious, because the right hand side
must have a double root at $x=1$ for the recurrence to make sense. The
first few terms are given by
\begin{equation*}
  \begin{gathered}
  F_1 = 1,\\
  F_2 = x^{2} + 2x + 1,\\
  F_3 = 4x^{4} + 8x^{3} + 10x^{2} + 8x + 4,\\
  F_4 = 34x^{6} + 68x^{5} + 94x^{4} + 104x^{3} + 94x^{2} + 68x + 34.
  \end{gathered}
\end{equation*}
The polynomial $F_n$ has degree $2n-2$ and palindromic
coefficients. 

The coefficients of these polynomials form the Poupard triangle
(\oeis{A8301}), first considered in 1989 by Christiane Poupard in the
article \cite{poupard_1989} and proved there to enumerate some kind of
labelled binary trees. It follows from this combinatorial
interpretation that all coefficients of $F_n$ are nonnegative
integers. For further combinatorial information on these polynomials
and their relatives, see \cite{FH_2013, FH_2014}.

The constant terms of these polynomials form the
sequence of \textit{reduced tangent numbers} (\oeis{A2105}), that can be defined for
$n \geq 1$ by the formula
\begin{equation}
  2^n (2^{2n} - 1) |B_{2n}| / n,
\end{equation}
where $B_n$ are the classical Bernoulli numbers, and starts by
\begin{equation*}
  1, 1, 4, 34, 496, 11056, 349504, 14873104, 819786496, \dots
\end{equation*}
One can deduce from \eqref{recu_F} that $F_{n+1}(0) = F_n(1)$, so the
reduced tangent numbers also describe the values of the polynomials
$F_n$ at $x=1$.

Our first result is the following unexpected property, that was the experimental starting point of this article.

\begin{Theorem}\label{th:rootsPoupard}
  For $n\geq 1$, all roots of the polynomial $F_n(x)$ are on the unit circle.
\end{Theorem}

This is proved in \cref{sec:roots} in a much more general context, by
showing that, for any positive integer $D$, a linear operator $\next{D}$
maps palindromic polynomials with nonnegative coefficients to
palindromic polynomials with nonnegative coefficients and all roots on
the unit circle.

\smallskip

As another interesting application, one can consider the sequence of polynomials
characterized by
\begin{equation}
  \label{recu_K}
  (x - 1)^2 G_{n+1}(x) = (x^{2n+3} + 1) G_{n}(1) - 2 x^2 G_n(x) \quad \text{for} \quad  n \geq 1,
\end{equation}
with initial condition $G_1 = 1 + x$. The first few terms are
\begin{equation*}
  \begin{gathered}
    G_1 = x + 1,\\
    G_2 = 2x^{3} + 4x^{2} + 4x + 2,\\
    G_3 = 12x^{5} + 24x^{4} + 32x^{3} + 32x^{2} + 24x + 12,\\
    G_4 = 136x^{7} + 272x^{6} + 384x^{5} + 448x^{4} + 448x^{3} + 384x^{2} + 272x + 136.
  \end{gathered}
\end{equation*}
The polynomial $G_n$ has degree $2n-1$ and palindromic coefficients. 

\begin{Theorem}\label{th:rootsKreweras}
  For $n\geq 1$, all roots of the polynomial $G_n(x)$ are on the unit circle.
\end{Theorem}

Because the polynomials $G_n$ have odd degree, they are all divisible
by $x+1$. One can also show by induction that the polynomial $G_n$ is
divisible by $2^{n-1}$. The quotient polynomials $2^{1-n} G_n/(x+1)$
have appeared in an article of Kreweras \cite{kreweras_1997} dealing
with refined enumeration of some sets of permutations. Their constant
terms are the Genocchi numbers (\oeis{A1469}), given by the formula
\begin{equation}
  2 (2^{2n} - 1) |B_{2n}|,
\end{equation}
where $B_n$ are again the Bernoulli numbers.

Both theorems above are proved in \cref{sec:roots} using a familly of
operators $\next{D}$ acting on palindromic polynomials. Section 3
describes explicit simple eigenvectors of the operator
$\next{1}$. In section 4, some evidence is given for the general
asymptotic behaviour of the iteration of the operators $\next{D}$ for
$D>1$. The last section contains various statements and conjectures on
values of the operators $\next{D}$ on specific palindromic polynomials.

\smallskip

Let us note as a side remark that another familly of polynomials, also
related to Bernoulli numbers, has been proved in \cite{lalin,
  lalin_addendum} to have only roots on the the unit circle, by
different methods.

\section{Operators $\next{D}$ and roots on the unit circle}

\label{sec:roots}

Let us consider a polynomial $P(x) = \sum_{j=0}^d p_j x^j$ with
rational coefficients. Let us say that the polynomial $P$ is
\textit{palindromic of index} $d$ if $p_j = p_{d-j}$ for all $j$. Note
that the index can also be described as the sum of the degree and the
valuation. For example, the index of the polynomial $x = 0 + x + 0 x^2$ is $2$. For
any $d\geq 0$, let $\vs{d}$ be the vector space spanned by palindromic
polynomials of index $d$.

For every nonnegative integer $D$, let us introduce a linear operator
$\next{D}$ from $\vs{d}$ to $\vs{d+2D-2}$. This
operator is characterized by the following formula:
\begin{equation}
  \label{def_next}
  (x - 1)^2 \next{D}(P)(x) = (x^{d + 2D} + 1) P(1) - 2 x^D P(x).
\end{equation}
The definition requires that the right hand side is divisible by
$(x-1)^2$. By linearity of \eqref{def_next}, it is enough to check
this property for the basis elements $x^{i}+x^{d-i}$ with $0 \leq i \leq d$, where one finds
\begin{equation}
  \label{comput}
  \next{D}(x^{i}+x^{d-i}) = 2 \frac{(1-x^{i+D})}{(1-x)}\frac{(1-x^{d+D-i})}{(1-x)},
\end{equation}
which is a polynomial with nonnegative integer coefficients. Note that
when $d = 2i$, one can divide \eqref{comput} by $2$.

The definition of $\next{D}$ and formula \eqref{comput} imply
immediately the following lemma.
\begin{Lemma}
  Let $P$ be a non-zero palindromic polynomial of index $d$ with
  nonnegative integer coefficients. If $d \leq 1$, assume moreover that
  $D > 0$. Then $\next{D}(P)$ is a non-zero palindromic polynomial
  of index $d+2D-2$ with positive integer coefficients.
\end{Lemma}

Let us record the following useful statement as a lemma.
\begin{Lemma}
  \label{power_of_two}
  When iterating $i$ times $\next{D}$ on an palindromic polynomial $P$
  of odd index with integer coefficients, the integer $2^i$ divides
  $\next{D}^i P$.
\end{Lemma}
\begin{proof}
  If the index of a palindromic polynomial $P$ is odd, then it is
  divisible by $x+1$. When $P$ has integer coefficients, formula
  \eqref{def_next} then implies that $\next{D}(P)$ has one further
  factor $2$. The lemma follows by induction.
\end{proof}

Recall that a palindromic polynomial $P=\sum_{j=0}^d p_j x^j$ is
called \textit{unimodal} if the sequence of coefficients is increasing
up to the middle coefficient(s), then decreasing. A polynomial $P$ is
called \textit{concave} if the piecewise linear function that maps $j$
to $p_j$ is a concave function. A concave polynomial $P$ is called
\textit{strictly concave} if every point $(j, p_j)$ is moreover an
extremal point in the graph of this piecewise linear function.

\begin{Lemma}
  Let $P$ be a non-zero palindromic polynomial of index $d$ with
  nonnegative integer coefficients. If $d \leq 1$, assume moreover that
  $D > 0$. Then $\next{D}(P)$ is unimodal and
  concave. If $P$ has no zero coefficient, then $\next{D}(P)$ is strictly concave.
\end{Lemma}
\begin{proof}
  By \eqref{comput}, the polynomial $\next{D}(P)$ is a nonnegative linear
  combination of unimodal and concave polynomials, hence itself
  unimodal and concave. Each term in \eqref{comput} gives two extremal
  points, or just one extremal point when $i = d-i$. When $P$ has no
  zero coefficient, this implies that there is an extremal point above
  every integer between $1$ and $d+1$.
\end{proof}

\smallskip

Let us now recall a beautiful criterion obtained by Lakatos and Losonczi
in \cite{lakatos_2004}.

\begin{Lemma}
  \label{lakatos}
  Let $P(x) = \sum_{j=0}^d p_j x^j$ be a palindromic polynomial of index $d$. If 
  \begin{equation}
    |p_d| \geq \frac{1}{2} \sum_{j=1}^{d-1} |p_j|,
  \end{equation}
 then all roots of $P$ are on the unit circle.
\end{Lemma}

From this criterion, one deduces another one.

\begin{Theorem}\label{th:main}
  Let $P(x) = \sum_{j=0}^d p_j x^j$ be a palindromic polynomial of index $d$. If
  \begin{equation}
    \label{eq:condition}
     2 p_j \geq p_{j-1} + p_{j+1} \quad \text{for all}\quad 0 \leq j \leq d,
  \end{equation}
  with the convention that $p_{-1} = p_{d+1} = 0$,
  then all roots of $P$ are on the unit circle.
\end{Theorem}
\begin{proof}
  Let $Q(x)=(1-x)^2 P(x)$. Then $Q(x) = \sum_{j=0}^{d+2} q_j x^j$,
  where
  \begin{align*}
    q_0&=  p_0,\\
    q_{j+1}&=   p_{j+1} + p_{j-1}-2p_j, \qquad (0\leq j \leq d) \\
    q_{d+2}&=  p_d.
  \end{align*}
  Note that $Q$ is also palindromic of index $d+2$.

  By the hypothesis \eqref{eq:condition}, all $q_j \leq 0$ for $1\leq j\leq d+1$.
  Since $Q(1)=0$, we have
  \begin{equation*}
    \sum_{j=1}^{d+1} |q_j|  =-\sum_{j=1}^{d+1} q_j = q_0 + q_{d+2} = 2 q_{d+2}.
  \end{equation*}
  Note that therefore $q_0 \geq 0$.
 
  Since $Q(x)$ is palindromic, and
  \begin{equation*}
    |q_{d+2}|  = \frac{1}{2} \sum_{j=1}^{d+1} |q_j|,
  \end{equation*}
  one can therefore apply \cref{lakatos} to $Q(x)$ and conclude that
  $Q(x)$ has all its roots on the unit circle. This implies the same
  property for $P(x)$.
\end{proof}

\begin{Theorem}
  \label{th:next_roots}
  Let $P$ be a non-zero palindromic polynomial of index $d$ with
  nonnegative integer coefficients. If $d \leq 1$, assume moreover
  that $D > 0$. Then $\next{D}(P)$ is a non-zero palindromic
  polynomial of index $d+2D-2$ with nonnegative integer coefficients and
  all roots of $\next{D}(P)$ are on the unit circle.
\end{Theorem}
\begin{proof}
  This is an application of \cref{th:main}. The definition of $\next{D}$
  and the hypothesis that $P$ has nonnegative coefficients imply
  immediately the condition \eqref{eq:condition}.
\end{proof}

Let us now apply \cref{th:next_roots} to the proofs of
\cref{th:rootsPoupard} and \cref{th:rootsKreweras}. The defining
recurrence \eqref{recu_F} for the polynomials $F_n$ can be written as
$F_{n+1} = \next{D}(F_n)$ with the initial condition $F_1 = 1$. The
property follows by induction. The same proof works for $G_n$ with the initial polynomial $1+x$.

\medskip

Let us now state two useful lemmas.
\begin{Lemma}
  \label{kernel}
  For all $d \geq 0$, the polynomial $x^d+1$ is in the kernel of $\next{0}$.
\end{Lemma}
\begin{proof}
  This is a direct consequence of \eqref{comput}.
\end{proof}

\begin{Lemma}
  \label{next0_qint}
  Let $d \geq 2$ be an integer. Then
  \begin{equation}
    \next{0}(1+x+\dots+x^d) = \sum_{i=0}^{d-2} (d-1-i)(i+1) x^i.
  \end{equation}
\end{Lemma}
\begin{proof}
  From the definition of $\next{0}$ by \eqref{comput}, and by the
  previous lemma, this is equal to
  \begin{equation*}
    \sum_{j=1}^{d-1} \frac{1-x^j}{1-x}\frac{1-x^{d-j}}{1-x}.
  \end{equation*}
  Developing, one find that the coefficient of $x^i$ is the cardinality of
  \begin{equation*}
    \left\{ (j,k) \mid 0\leq k \leq j-1 \quad\text{and}\quad 0\leq i-k \leq d-j-1\right\}.
  \end{equation*}
  But this is the same as the set
  \begin{equation*}
    \left\{ (j,k) \mid 1\leq j-k \leq d-i-1 \quad\text{and}\quad 0\leq k \leq i \right\},
  \end{equation*}
  whose cardinality is $(d-1-i)(i+1)$.
\end{proof}

\section{Sinus polynomials as eigenvectors}

As can be seen in the right picture of \cref{figure_F12}, the roots of
the Poupard polynomials $F_n(x)$ are very close to some of the roots
of $x^{2n} + 1$, with two missing roots on the right. Moreover the
plot of the coefficients of $F_n(x)$ seem to approximate a concave continuous
function, as in the left picture of \cref{figure_F12}.

\begin{figure}
  \centering
  \includegraphics[width=0.5\textwidth]{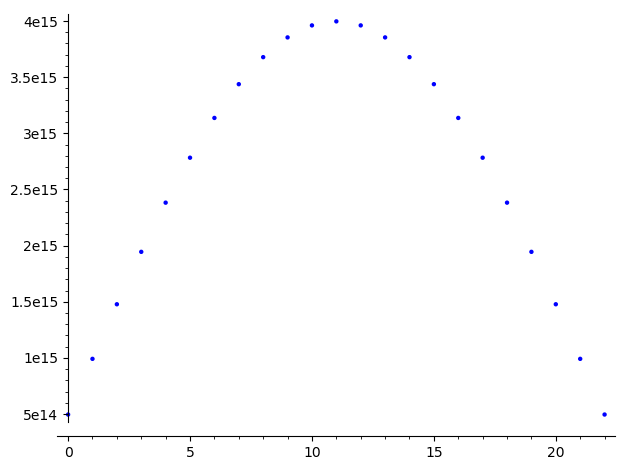}\includegraphics[width=0.5\textwidth]{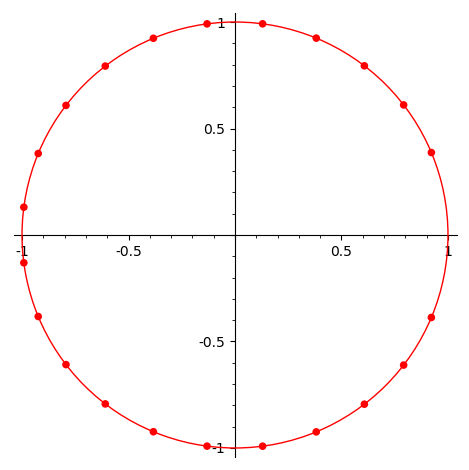}
  \caption{Coefficients and roots of the Poupard polynomial $F_{12}$}
  \label{figure_F12}
\end{figure}

One expects that, up to a global multiplicative factor, the
polynomials obtained when iterating $n$ times the operator $\next{D}$
(for some fixed $D > 1$) are always becoming, when $n$ is large, very
close to the polynomials described in this section. Some kind of
justification will be given in the next section.

\medskip

Let us consider the polynomial $S_{m,n}(x)$ defined for $n \geq 2$ and
odd $m \geq 1$ by
\begin{equation}
  \label{defi_Sn}
  S_{m,n}(x) = \frac {x^{m n}+1}{x^2 - 2x \cos{\frac{\pi}{n}}+1},
\end{equation}
whose roots are the roots of $x^{m n}+1$ except $\exp({\frac{i\pi}{n}})$ and its conjugate.

Let us first give an alternative expression for $S_{m,n}$.
\begin{Lemma}
  \label{explicit_S}
  The polynomial $S_{m,n}$ has an explicit expression
  \begin{equation}
    S_{m,n}(x) = \frac {1}{\sin(\frac {\pi}{n})}  \sum_{k=0}^{m n-2}  {\sin\left(\frac {(k+1)\pi}{n}\right) } x^{k}.
\end{equation}
\end{Lemma}
\begin{proof}
  The proof is a simple computation, expanding both sides as polynomials in $x$
  and $\zeta = \exp({\frac{i\pi}{n}})$, also using that $m$ is odd.
\end{proof}

This implies that the plot of the coefficients of $S_{1,n}$ looks very
much like a sinus curve, like the left image in \cref{figure_F12}.

\begin{Proposition}
  For every $n \geq 2$ and odd $m \geq 1$, the polynomial $S_{m,n}$ is
  an eigenvector of the operator $\next{1}$ acting on $\vs{m n-2}$, for the
  eigenvalue $1 / (1 - \cos(\pi/n))$.
\end{Proposition}
\begin{proof}
  The proof is another explicit computation using the definition of $S_{m,n}$
  in \eqref{defi_Sn} and the definition of the operator $\next{1}$ in
  \eqref{def_next}.
\end{proof}

Note that the eigenvalue is also the value $S_{m,n}(1)$.

In general, the Galois conjugates of the polynomial $S_{1,n}$ are not
providing a complete set of eigenvectors for the operator $\next{1}$
acting on $\vs{n - 2}$. The other eigenvectors are $S_{m,n/m}$ for
odd divisors $m$ of $n$, and their Galois conjugates.

The familly of operators $\next{1}$ acting on the spaces $\vs{n-2}$ of
palindromic polynomials looks very much like discrete versions of the
Laplacian operator $\partial_x^2$ acting on the space of functions $f$
on the real interval $[0,1]$ such that $f(1-x)=f(x)$ for all $x$ and
$f(0) = f(1) = 0$.

\section{Asymptotic behaviour from recurrence}

Our next point is to justify in a heuristic way that iterating an
operator $\next{D}$ for some $D > 1$ produces a sequence of
polynomials that gets closer and closer to the sinus polynomials
$S_{1,n}$. We have not tried to make these computations rigourous.

Let us consider a familly of polynomials $H_n$ of index $n$ defined by
iterating $\next{D}$, starting from an arbitrary palindromic
polynomial $H_m$ with nonnegative coefficients and index $m$. In all
this section, the index $n$ belongs to an arithmetic progression
of step $\delta=2D-2$ starting at $m$. Let us denote
\begin{equation}
  H_n(x) = \sum_{k=0}^{n} H_{n,k} x^k.
\end{equation}

We will assume the following asymptotic ansatz for the constant terms:
\begin{equation}
  \label{ansatz_Fn0}
  H_n(0) \simeq \cs{A} \cs{B}^{n} n^\cs{C} n^{\cs{E}n},
\end{equation}
for some constants $\cs{A},\cs{B},\cs{C},\cs{E}$, with
$\cs{A},\cs{B},\cs{E}$ positive. This ansatz is motivated by the known
case of the tangent numbers, where $\cs{B}=\frac{2}{e\pi}$,
$\cs{E}=1$ and $\cs{C}=-1/2$. This ansatz implies that
\begin{equation}
  H_{n+\delta}(0) / H_n(0) \simeq \cs{B}^{\delta} e^{\delta\cs{E}} n^{\delta\cs{E}}.
\end{equation}

We will also assume that there exists a smooth function $\Psi$ which
is a probability distribution function on the real interval $[0,1]$
vanishing at $0$ and $1$, with $\Psi(1-x) = \Psi(x)$ on this interval
and such that
\begin{equation}
  \label{ansatz_Fnk}
  H_{n,k} \simeq \frac{\alpha_n}{n}\Psi\left(\frac{k}{n}\right) + H_{n,0}
\end{equation}
is a good asymptotic approximation when $n$ is large, for some
sequence $\alpha_n$ to be determined.

Taking the sum of \eqref{ansatz_Fnk} over $k$ ranging from $0$ to $n$ and using the
hypothesis on $\Psi$, one gets
\begin{equation*}
  H_{n+\delta,0} = H_n(1) \simeq \alpha_n + (n + 1) H_{n,0},
\end{equation*}
Assuming that $n H_{n,0}$ is negligible compared to $H_{n+\delta,0}$, one
obtains that a correct choice for $\alpha_n$ is
\begin{equation*}
  \alpha_n = H_{n+\delta,0}.
\end{equation*}

From \eqref{def_next}, one deduces that the action of $\next{D}$ at the level of coefficients is given by
\begin{equation}
  \label{recu_coeffs}
  H_{n+\delta,k+D} - 2 H_{n+\delta,k+D-1} + H_{n+\delta,k+D-2} = -2 H_{n,k},
\end{equation}
except for $k=0$ and $k=n$.

Replacing in \eqref{recu_coeffs} the coefficients by the expression from
\eqref{ansatz_Fnk}, one obtains
\begin{multline}
  \frac{\alpha_{n+\delta} }{n+\delta}\left( \Psi\left(\frac{k+D}{n+\delta}\right)-2 \Psi\left(\frac{k+D-1}{n+\delta}\right) +\Psi\left(\frac{k+D-2}{n+\delta} \right)\right) \\\simeq -2 \left(\frac{\alpha_n}{n} \Psi\left(\frac{k}{n}\right) + H_{n,0}\right) 
\end{multline}

Using now the growth ansatz, one can get rid of $H_{n,0}$ in
the rightmost term and obtain
\begin{equation}
  \Psi\left(\frac{k+D}{n+\delta}\right)-2 \Psi\left(\frac{k+D-1}{n+\delta}\right) +\Psi\left(\frac{k+D-2}{n+\delta} \right) \\\simeq -2 \frac{\alpha_n}{\alpha_{n+\delta}} \Psi\left(\frac{k}{n}\right).
\end{equation}
The left hand side is an approximation of the second derivative of
$\Psi$, so that one obtains
\begin{equation}
  \frac{1}{2 (n+\delta)^{2}} \Psi''\left(\frac{k}{n+\delta}\right) \simeq -2 \frac{\alpha_n}{\alpha_{n+\delta}} \Psi\left(\frac{k}{n}\right).
\end{equation}
If $\delta \cs{E}=2$, one therefore reaches the following
differential equation
\begin{equation}
  \Psi'' = - \cs{F} \Psi,
\end{equation}
where $\cs{F} = \frac{4}{\cs{B}^{\delta}e^2}$. Because $\Psi$ vanishes at $0$, it must
be a multiple of $\sin(\sqrt{\cs{F}} x)$. Because $\Psi$ vanishes at $1$
and is positive on the interval $[0,1]$, necessarily $\cs{F} =
\pi^2$ and therefore $\cs{B}^{\delta}=(\frac{2}{e\pi})^2$. Because $\Psi$ is a probability distribution, one must have
$\Psi = \frac{\pi}{2} \sin(\pi x)$.

One can therefore conclude that, under several plausible but unproven
assumptions, the asymptotic shape of the coefficients of the
polynomials $H_n$ is approximating that of the polynomials $S_{1,n+2}$.








\section{Various remarks}

\subsection{Action of the operator $\next{0}$.}

Applying the operator $\next{0}$ decreases the index by $2$, so that
iterating this operator on any initial polynomial $P$ of index $d$
always vanishes after a finite number of steps. Let $\next{0}^{\max}$ be the last
non-identically zero iterate of $\next{0}$ acting on $\vs{d}$. Let us
denote by $\rho$ the linear map that maps $P$ to the constant term of
$\next{0}^{\max}(P)$.

For example, here is a sequence of iterates of $\next{0}$:
\begin{equation*}
  x^{4} + x^{3} + x^{2} + x + 1, \quad 3x^{2} + 4x + 3,\quad  4.
\end{equation*}
In this case, $\rho(x^{4} + x^{3} + x^{2} + x + 1) = 4$.

Let us present some special cases of initial choices where the value
of $\rho$ is interesting.

\medskip

For $n\geq 0$, consider the polynomial
\begin{equation}
  Q_n(t) = \sum_{i=0}^{2n+1} \rho\left(\frac{x^i-x^{2n+1-i}}{x-1}\right) t^i,
\end{equation}
recording this sequence of final values. By antisymmetry of the
argument of $\rho$, the polynomial $Q_n$ vanishes at $t=1$. Let
$P_n(t)$ be the quotient $Q_n(t) / (t-1)$, which is clearly a
palindromic polynomial.

\begin{Proposition}
  For every $n \geq 0$, the polynomial $P_n$ is the Poupard polynomial $F_{n+1}$.
\end{Proposition}
\begin{proof}
  For $n=0$, one can check that $P_n(t) = 1$. Assume $n > 0$. For
  $0 \leq i \leq 2n$, the coefficient $c_{n,i}$ of $t^i$ in $P_n(t)$
  can be written as
  \begin{equation}
    \label{coeff_c}
    - \rho\left( \sum_{0 \leq k \leq i} \frac{x^k-x^{2n+1-k}}{x-1} \right) =
   \rho\left(    \frac{x^{i+1}-1}{x-1} \frac{x^{2n+1-i}-1}{x-1} \right).
  \end{equation}
  Let us now compute $c_{n,i+2} - 2c_{n,i+1}+c_{n,i}$ for
  $0 \leq i \leq 2n - 2$.  Starting from the left hand side of
  \eqref{coeff_c}, this is given by
  \begin{equation*}
    \rho \left(x^{i+1} + x^{2n-i-1} \right).
  \end{equation*}
  Using now the equation \eqref{comput} for $\next{0}$ and the definition of
  $\rho$ as the final value for the iteration of $\next{0}$, this
  becomes
  \begin{equation*}
    2 \rho \left( \frac{x^{i+1}-1}{x-1}\frac{x^{2n-i-1}-1}{x-1} \right),
  \end{equation*}
  in which one can recognize $-2 c_{n-1, i}$ using the right hand side
  of \eqref{coeff_c}.

  Moreover, $c_{n,1} - 2 c_{n,0} = \rho( 1 + x^{2n}) = 0$ because
  $1 + x^{2n}$ is in the kernel of $\next{0}$ by \cref{kernel}.

  Let us now check that $c_{n,0} = \sum_{i=0}^{2n-2}
  c_{n-1,i}$. First, by \eqref{coeff_c}, the left hand side is the
  image by $\rho$ of $\next{0}(1+x+\dots+x^{2n})$, given by
  \Cref{next0_qint}. The right hand side is the image by $\rho$ of
  \begin{equation}
    \sum_{i=0}^{2n-2} \sum_{0 \leq k \leq i} \sum_{k \leq j \leq 2n-2-k} x^j =
    \sum_{j=0}^{2n-2} (2n-1-j)(j+1) x^j,
  \end{equation}
  which is the exact same expression.

  All these properties of the coefficients $c_{n,i}$ imply exactly
  that the polynomial $P_{n}(t)$ is the image of $P_{n-1}(t)$ by
  $\next{1}$, acting on the variable $t$.
\end{proof}


\medskip

For $n\geq 0$, consider the polynomial
\begin{equation}
  Q'_n(t) = \sum_{i=0}^{2n} \rho\left(\frac{x^i-x^{2n-i}}{x-1}\right) t^i,
\end{equation}
recording this sequence of final values. By antisymmetry of the
argument of $\rho$, the polynomial $Q'_n$ vanishes at $t=1$. Let
$P'_n(t)$ be the quotient $Q'_n(t) / (t-1)$, which is clearly a
palindromic polynomial of odd index.

\begin{Proposition}
  For every $n \geq 1$, the polynomial $Q'_n$ is the Kreweras
  polynomial $G_{n}$.
\end{Proposition}
\begin{proof}
  The proof is very similar to the previous one. One first check that
  $Q'_1$ is $1+x$. Then one checks by looking at coefficients that
  $Q'_{n+1}$ is $\next{1}(Q'_n)$.
\end{proof}




\smallskip

Let us now describe some similar conjectural properties. For the
starting sequence $(2^{-j} (1+x)^{2j})_{j \geq 0}$, one gets the
following values of $\rho$:
\begin{equation*}
  1, 1, 5, 61, 1385, 50521, 2702765, 199360981, 19391512145, \dots
\end{equation*}
which seem to be the Euler numbers \oeis{A364}. Similarly, for the
starting sequence $(2^{-j} (1+x)^{2j+1})_{j \geq 0}$, one gets
\begin{equation*}
  1, 3, 25, 427, 12465, 555731, 35135945, \dots
\end{equation*}
This seems to be the closely related sequence \oeis{A9843}.


As a final conjectural remark, let us consider the following extension of the two previous cases.
\begin{Conjecture}
  For every $i,j$, the number $\rho(x^i(1+x)^j)$ is divisible by $2^{\lfloor j/2 \rfloor}$.
\end{Conjecture}
This property is clear if $j$ is odd by \cref{power_of_two}, but not
at all if $j$ is even.

Assuming this conjecture, one can define, for every integer $n$, the
square matrix $\mathbf{M}_n$ whose coefficient $\mathbf{M}_n(i,j)$,
for $0 \leq i \leq n$ and $0 \leq j \leq n$, is
$\rho(x^i(1+x)^j)2^{-\lfloor j/2 \rfloor}$. 
\begin{Conjecture}
For all $n\geq 0$, the determinant $d_{n}$ of the matrix $\mathbf{M}_n$ is given by the formula
\begin{equation}
  \label{nice_det}
  d_n = (n-1)!^{\epsi(1)} (n-2)!^{\epsi(2)} (n-3)!^{\epsi(3)} \dots 1!^{\epsi(n-1)},
\end{equation}
where
\begin{equation*}
  \epsi(k) = \begin{cases}
    2 \quad\text{if $k$ is odd,}\\
    4 \quad\text{if $k$ is even.}
    \end{cases}
\end{equation*}
\end{Conjecture}

For example, $\mathbf{M}_6$ is equal to
\begin{equation}
\left(\begin{array}{rrrrrr}
1 & 1 & 1 & 3 & 5 & 25 \\
1 & 2 & 3 & 14 & 33 & 226 \\
2 & 8 & 18 & 120 & 378 & 3336 \\
10 & 64 & 198 & 1728 & 6858 & 74304 \\
104 & 896 & 3528 & 38016 & 182088 & 2339712 \\
1816 & 19456 & 92808 & 1188864 & 6668568 & 99118080
\end{array}\right)
\end{equation}
whose determinant is indeed $5!^2 4!^4 3!^2 2!^4 1!^2$.

This matrix contains entries with large prime factors, for example $92808 = 2^3  3^2  1289$, but the determinant has only small prime factors.

\subsection{Action of the operator $\next{1}$}

Applying the operator $\next{1}$ does not change the index, so
iterating this operator on any initial choice gives an infinite
sequence of palindromic polynomials of the same index.

For example, starting with $x$ gives a sequence of polynomials
\begin{equation*}
x, x^{2} + 2x + 1, 4x^{2} + 6x + 4, 14x^{2} + 20x + 14, 48x^{2} + 68x + 48, 164x^{2} + 232x + 164, \dots
\end{equation*}
whose constant terms and middle coefficients are given by \oeis{A7070}
and by \oeis{A6012}. Indeed, the action of $\next{1}$ on reciprocal
polynomials of index $2$ is given in the basis $\{1+x^2,x\}$ by the
matrix
\begin{equation*}
  \left(\begin{array}{rr}
          2 & 1 \\
          2 & 2
        \end{array}\right)
\end{equation*}
so that both sequences satisfy the recurrence
$a_n = 4 a_{n-1} - 2 a_{n-2}$ with appropriate initial conditions.

\subsection{Action of the operator $\next{2}$.}

Applying the operator $\next{2}$ increases the index by $2$, so
iterating this operator gives an infinite sequence of polynomials for
every initial choice. In each such sequence, the sequence of constant
terms is, up to a shift of indices by one, the same as the sequence of
values at $x=1$. Some examples were presented in the introduction,
related to reduced tangent numbers and Genocchi numbers. Let us record
one more familly of examples.

Using the polynomials $x^i(x+1)$ for $i \geq 0$ as starting points,
one gets a table of constant terms:
\begin{equation*}
\left(\begin{array}{rrrrrr}
1 & 1 & 3 & 17 & 155 & 2073 \\
0 & 1 & 6 & 55 & 736 & 13573 \\
0 & 1 & 10 & 135 & 2492 & 60605 \\
0 & 1 & 15 & 280 & 6818 & 211419 \\
0 & 1 & 21 & 518 & 16086 & 619455 \\
0 & 1 & 28 & 882 & 34020 & 1592811
\end{array}\right)
\end{equation*}
Here $i$ is the row index and in each row the term of index $j$ is
the constant term divided by $2^j$. This table seems to be essentially the Salié
triangle \oeis{A65547}.

\bibliographystyle{plain}
\bibliography{rootpoupard.bib}

\begin{thebibliography}{1}

\bibitem{FH_2013}
D.~Foata and G.-N. Han.
\newblock Finite difference calculus for alternating permutations.
\newblock {\em J. Difference Equ. Appl.}, 19(12):1952--1966, 2013.

\bibitem{FH_2014}
D.~Foata and G.-N. Han.
\newblock Tree calculus for bivariate difference equations.
\newblock {\em J. Difference Equ. Appl.}, 20(11):1453--1488, 2014.

\bibitem{kreweras_1997}
G.~Kreweras.
\newblock Sur les permutations compt\'{e}es par les nombres de {G}enocchi de
  1-i\`ere et 2-i\`eme esp\`ece.
\newblock {\em European J. Combin.}, 18(1):49--58, 1997.

\bibitem{lakatos_2004}
P.~Lakatos and L.~Losonczi.
\newblock Self-inversive polynomials whose zeros are on the unit circle.
\newblock {\em Publ. Math. Debrecen}, 65(3-4):409--420, 2004.

\bibitem{lalin}
M.~N. Lal\'{\i}n and C.~J. Smyth.
\newblock Unimodularity of zeros of self-inversive polynomials.
\newblock {\em Acta Math. Hungar.}, 138(1-2):85--101, 2013.

\bibitem{lalin_addendum}
M.~N. Lal\'{\i}n and C.~J. Smyth.
\newblock Addendum to: {U}nimodularity of zeros of self-inversive polynomials [
  3015964].
\newblock {\em Acta Math. Hungar.}, 147(1):255--257, 2015.

\bibitem{poupard_1989}
C.~Poupard.
\newblock Deux propri\'{e}t\'{e}s des arbres binaires ordonn\'{e}s stricts.
\newblock {\em European J. Combin.}, 10(4):369--374, 1989.

\end{thebibliography}

\end{document}